\documentclass[11pt,a4paper,twoside]{article}
\usepackage[utf8]{inputenc}

\usepackage[a4paper, total={6in, 8in}]{geometry}

\usepackage{xcolor}
\usepackage{amsmath,amssymb,amsthm,amsfonts}

\usepackage{tikz}
\usepackage{pgfplots}
\pgfplotsset{compat=1.11}
\usetikzlibrary{arrows.meta}
\input{arrowsnew}

\usepackage{pgfplots}
\pgfplotsset{width=7.5cm,compat=1.16}
\usetikzlibrary{positioning}
\usetikzlibrary{shapes.multipart,matrix,positioning}

\usepackage{bbm}
\usepackage{mathtools, nccmath, relsize}
\usepackage{hyperref}
\usepackage{authblk}

\theoremstyle{definition}
\newtheorem{theorem}{Theorem}[section] 
\newtheorem{definition}[theorem]{Definition}
\newtheorem{corollary}[theorem]{Corollary}
\newtheorem{observation}[theorem]{Observation}

\newtheorem{proposition}[theorem]{Proposition}

\newtheorem{example}[theorem]{Example}

\newtheorem{remark}[theorem]{Remark}

\newcommand{\inte}[1]{\boldsymbol{#1}}
\newcommand{\lo}[1]{\underline{#1}}
\newcommand{\hi}[1]{\overline{#1}}
\newcommand{\iunit}[1]{\textcolor{red}{#1}}

\title{Exact bounds on the amplitude and phase of the interval discrete Fourier transform in polynomial time}
\author{Marco de Angelis}

\date{December 17, 2021}

\begin{document}

\maketitle

\begin{abstract}
We elucidate why an interval algorithm that computes the exact bounds on the amplitude and phase of the discrete Fourier transform can run in polynomial time.  We address this question from a formal perspective to provide the mathematical foundations underpinning such an algorithm. We show that the procedure set out by the algorithm fully addresses the dependency problem of interval arithmetic, making it usable in a variety of applications involving the discrete Fourier transform. 

For example when analysing signals with poor precision, signals with missing data, and for automatic error propagation and verified computations.

\end{abstract}

\vspace{1em}
\noindent
Keywords:\ Interval mathematics, Dependency tracking, Discrete Fourier transform.

\section{Introduction}\label{sec:intro}
In a recent work, an algorithm that computes the exact bounds on the amplitude of the discrete Fourier transform (DFT) in polynomial time was presented \cite{intervalfourier2020}. In this paper, the focus is shifted towards the mathematical foundations underpinning the interval algorithm presented in the above work.

The paper is structured as follows: first, we introduce the mathematical expression of the Fourier transform under study, see §\ref{sec:dft}.  In §\ref{sec:theory} propaedeutic theoretical foundations of interval analysis are given, in §\ref{sec:extensions} the interval Fourier transform is introduced, in §\ref{sec:reach} the main results are presented, followed by a discussion section in §\ref{sec:discussion} .

\subsection*{Notation}

The standardized notation of interval analysis is adopted throughout the manuscript \cite{kearfott2005standardized}.

\subsection{The discrete Fourier transform} \label{sec:dft}
Let $x=(x_0,\ldots, x_{N-1}) \in \mathbb{R}^N$ be a vector, denote $x_n$ the $(n+1)$-th coordinate of the vector $x$.  We study the discrete Fourier transform 
of  $x$ which is given by the algebraic expression:

\begin{equation}\label{eq:algebraic_dft}
\mathsf{f}_h(x) := \sum_{n=0}^{N-1} x_n \  w_{hn} 
\end{equation}
%e^{-\textcolor{red}{i} \frac{2 \pi}{N} h n}
\noindent
where $w_{hn} :=e^{-\textcolor{red}{i} \frac{2 \pi}{N} h n}$ is the so-called Fourier coefficient with $\textcolor{red}{i}=\sqrt{-1}$, $w_{hn} \in \mathbb{C}$, and $h=0,...,N-1$  a positive integer often called a harmonic or frequency number. Since $n,h=0,\ldots, N-1$, the Fourier coefficients form a matrix  of size $N\times N$  called $ W$, with elements $w_{hn}:=W[h,n]$. 

\begin{figure}[h!]
\centering
\begin{tikzpicture}[cell/.style={rectangle,draw=black}, space/.style={minimum height=0.6em, minimum width=1em,matrix of nodes,row sep=-\pgflinewidth,column sep=-\pgflinewidth, column 1/.style={font=\ttfamily}},text depth=0.5ex,text height=2ex,nodes in empty cells]

%\node[] (repeat) at (-2,4.9) {{\small $\forall \ j  = 1, \cdots, N_e $}}; 

%\node[] (W) at (-5,-0.5) {{ $ W = $}};
\node[] (wh) at (5,-0.7) {{ $ w_{h} $}};

\node[] (W) at (-5,-0.25) {$W=$};

\draw[gray, very thin, fill=gray!2!] (-3.4,-2.5) rectangle (4.5,1.6);
\draw[black, thick,] (-3.5,-1) rectangle (4.6,-0.4);
%\node[] (dimA) at (-2.,1.75) {{\small $n_a \times N_a \times N_e$}};
	\matrix (Amatrix) [font=\small,space] at (0,0)
	{
	     							             & \tiny $0$ & \tiny $1$ & $ \cdots$  & \tiny $n$ & $\cdots$ & \tiny $N-1$\\
	  {\tiny $0$}              & $w_{00}$ & $w_{01}$ & $ \cdots$  & $w_{0n}$ & $\cdots$ & $w_{0(N-1)}$  \\
	   {\tiny $1$}              & $w_{10}$ & $w_{11}$ & $ \cdots$  & $w_{1n}$ & $\cdots$ & $w_{1(N-1)}$  \\
	  {\tiny $\vdots$} &   &   & {\tiny $\vdots$}   &  &  \\
	  {\tiny $h$}          & $w_{h0}$ & $w_{h1}$ &  $ \cdots$ & $w_{hn}$ & $ \cdots$ & $w_{h(N-1)}$  \\
	  {\tiny $\vdots$} &   &  &  {\tiny $\vdots$}  &   &    \\
	  {\tiny $N-1$}      & $w_{(N-1)0}$ & $w_{(N-1)1}$ &  $ \cdots$ & $w_{(N-1)n}$ & $ \cdots$ & $w_{(N-1)(N-1)}$ \\
	};    

\end{tikzpicture}
\end{figure}

\noindent
Therefore \eqref{eq:algebraic_dft} can 
be seen as the complex inner product:
\begin{equation}\label{eq:compact_dft}
\mathsf{f}_h(x) :=  \sum_{n=0}^{N-1} x_n \ w_{hn}= \left<  w_h, x \right> = \overline{w}_h x^T
\end{equation}
where, $\overline{w}_h \in \mathbb{C}^N$, is the complex conjugate vector of Fourier coefficients at harmonic $h$, or  $(h+1)$-th row of $W$. 
Since the matrix of Fourier coefficients is symmetric and orthogonal, the DFT need not be evaluated for all $h=0,...,N-1$. In particular, when $N$ is a power of two, it will be sufficient to compute $\mathsf{F}_h$ for all $0 \leq h \leq N/2$. Also, observe that expressions \eqref{eq:algebraic_dft} and \eqref{eq:compact_dft} are equivalent from an interval analysis perspective because of repetitions invariance.\\

%inner product of \eqref{eq:compact_dft}

\begin{remark}
In what follows, the Fourier coefficients $w_{hn}$ will be considered as non-interval floating-point numbers. While this leads to results that are not verified, this assumption need not affect the generality of the presented method, which is primarily developed to deal with input uncertainties that are typically orders of magnitude larger in width than round-off errors. Under this working assumption, the presented interval algorithm can be considered rigorous but not verified.
%A verified version of the interval discrete Fourier transform is currently under development.
\end{remark}

\clearpage

\section{Theory of interval computations}\label{sec:theory}

In this section we will provide an overview of the fundamental definitions and theorems of interval computations that will be utilised to obtain the main result.\\

The theory of interval computations is based on two fundamental pillars: (i) the notion of an \emph{interval-valued function}, (ii) and the property of \emph{inclusion monotonicity}. An \emph{interval-valued function} is a function (or computer program), whose inputs are intervals. An interval-valued function is \emph{inclusion monotonic} if, when evaluated on two nested intervals, its resulting ranges are also nested with preserving order. Notion (i) implies that intervals represent a particular instance of a number; while property (ii) ensures that interval computations are inclusive thus rigorous. An interval-valued function that satisfies the property of inclusion monotonicity is called \emph{inclusive}. There are examples in the literature of interval-valued function that are not inclusion monotonic. Calculations that are not inclusion monotonic cannot be regarded as rigorous.
%an \emph{interval extension}

\subsection{Intervals and united extension}

\begin{definition}(interval)\label{def:interval}\\
Let $a, b \in \mathbb{R}$. An interval, denoted by $\inte{x}$, is the compact set: 
$$\inte{x}:=\{x : a \leq x \leq b \}.$$
An interval is said to be an element of an interval vector space $\mathbb{IR}$, such that $\inte{x} \in \mathbb{IR}$.
\end{definition}

\begin{definition}(n-interval or n-box)\label{def:N-interval}\\
Let $a, b \in \mathbb{R}^n$. An n-interval or interval vector, denoted by $\inte{x}$, is the n-box: 
$$\inte{x}:=\{x : a_i \leq x_i \leq b_i, \ \text{for} \ i=1,...,n \}.$$
An n-interval is said to be an element of an interval vector space $\mathbb{IR}^n$, such that $\inte{x} \in \mathbb{IR}^n$. 
\end{definition}

Note that our notation intentionally does not distinguish an n-interval from an interval.

Note that the definition of an n-interval implies non-interactivity, see Definition \ref{def:noninteractive}.

\begin{definition}(Interval-valued function)\label{def:interval-valued}\\
An \emph{interval-valued function} $F: \mathbb{IR}^n \rightarrow \mathbb{IR}^m$ is a function that maps an n-interval to an m-interval.

\end{definition}

%Note that an interval-valued function need not output an interval. 
Also note that $F$ is not inclusion monotonic in general.\\ %If, in what follows, a distinction is needed between the two 

%Example of an $F(\inte{x})$ that is not inclusion monotonic...

\begin{definition}(Inclusion monotonicity)\label{def:inclusion}\\
Let $\inte{x}, \inte{x}' \in \mathbb{IR}^n$, let $F:  \mathbb{IR}^n \rightarrow  \mathbb{IR}^m$ be an interval-valued function that outputs the two m-intervals: $\inte{y}=F(\inte{x})$, and $\inte{y}'=F(\inte{x}')$. Then $F$ is said to be \emph{inclusion monotonic} for $ \inte{x}' \subset \inte{x}$, if and only if: 
$$F(\inte{x}') \subset F(\inte{x}).$$
%holds true for all $\inte{x}, \inte{x}' \in \mathbb{IR}^n$, such that $\inte{x}' \subset \inte{x}$.
\end{definition}

\begin{definition}(Rigorous computations)\label{def:rigorous}\\
Interval computations are said to be \emph{rigorous} if they satisfy the property (ii) of inclusion monotonicity. 
\end{definition}

Note that Definition \ref{def:rigorous} is a weak definition of rigorous computing as it suitable only for (engineering) problems wherein input uncertainty is predominant. Stronger definitions of rigorous computing may include either/both round-off errors or/and truncation errors. 

\begin{definition}(Verified computations)\label{def:verified}\\
Interval computations are said to be \emph{verified} if they are rigorous and implemented using outward directed rounding. 
\end{definition}

\begin{definition}(United extension)\label{def:united_extension}\\
Let $\inte{x} \in \mathbb{IR}^n$ be an n-interval,  $A \subset \mathbb{R}^n$ an arbitrary set, $f: A \rightarrow B$  an arbitrary function that maps $A$ to $B \subset \mathbb{R}^m$, and let $S(A)$ be the collection of all n-intervals that are subsets of $A$. Then, the \emph{united extension} of $f$ is:
%Denote by $S(A)$ and $S(B)$ the family of subsets of $A$ and $B$, respectively. The set-valued mapping $\hat{f} : S(A) \rightarrow S(B)$
\begin{equation}\label{eq:united_extension}
\mathcal{F} (A) = \bigcup_{\inte{x} \in  S(A)} \left\{ f(x) :  x \in \inte{x} \right\}.
\end{equation}
%is the \emph{united extension} of $f$. Thus, $\hat{f}(X)$ is the union in $A$ of all the sets containing a single element $f(x)$ for some $x$ in $X$. 
\end{definition} 

\begin{definition}(United set)\label{def:united_set}\\
The image of the united extension of Definition \ref{def:united_extension} is also called a \emph{united set}, and is denoted by $B := \mathcal{F} (A)$.
\end{definition}

\begin{definition}
When $A=\inte{x}$ is an n-box rather than an arbitrary set the Definition \ref{def:united_extension} specialises to:
$$\mathcal{F} (\inte{x}) = \bigcup_{\inte{x}' \in  S(\inte{x})} \left\{ f(x) :  x \in \inte{x}' \right\}.$$
\end{definition}

Note that the united set of $\mathcal{F}$ denoted by $Y:=\mathcal{F} (\inte{x})$ is not an m-box. The range of $Y$ is often called a \emph{united box} and is denoted by:
$$\hat{\inte{y}} := \text{range}(Y).$$

\begin{observation}\label{prop:united_inclusion}
A united extension is inclusion monotonic.
%Let $\inte{x}, \inte{x}' \in \mathbb{IR}^n$, with $\inte{x}' \subset \inte{x}$, let $\mathcal{F}$ be a united extension. 
\end{observation}
\begin{proof}
Let $\mathcal{F}$ be a united extension of the function $f$, and let $A, A' \subset \mathbb{R}^n$ be two arbitrary sets such that $A' \subset A$. Then, from Definition \ref{def:united_extension} it clearly follows that: 
$$\mathcal{F}(A') \subset \mathcal{F}(A).$$
Hence $\mathcal{F}$ is inclusion monotonic.
%The proof follows from the definition of a function $f$, which can be found in any mathematics textbooks, e.g. \cite{function}.
\end{proof}

A united extension is abstracted of the notion of interval and interval computations, as it belongs to the more general class of set-valued function. %On the other hand, in order for an interval extension to be defined, the notions of interval arithmetic and interval computations are needed. \\
Moreover, a united extension and its united set are unique, thus there is only one single united extension of a function $f$, and one single united set of $f$ given 
 an input $A$ or $\inte{x}$.

Note that there is no polynomial-time algorithm that can compute a united extension in the general case.\\

The following result provides the fundamental condition for an interval-valued function to be considered inclusion monotonic. A proof can be found in \cite{moore1966interval}.
\begin{proposition}
Let $F: \mathbb{IR}^n \rightarrow \mathbb{IR}^m$ be interval-valued function, $\inte{x} \in \mathbb{IR}^n$ be a n-interval. Then, $F$ is said to be inclusion monotonic if:
$$\lim_{\inte{x} \rightarrow \hat{x}} F(\inte{x}) = F(\hat{x}) \ \ \text{for\ all}  \ \ \hat{x} \in \inte{x}.$$
%$$F(\inte{x}) = F(\hat{x}) \ \text{when} \ \inte{x} \rightarrow \hat{x}.$$
%\text{if} \ \text{and} \ \text{only} \ \text{if}
% $$F(\inte{x}) = F(\hat{x}) \ \text{when} \ \inte{x} \rightarrow \hat{x}.$$
\end{proposition}

\subsection{Interval arithmetic}
With the rules of interval arithmetic we can practically turn an algebraic expression or computer program in to an interval-value function. 

\begin{definition}(United extension of a binary operation)\\
Let $\inte{x}, \inte{y} \in \mathbb{IR}$, $S(\inte{x}, \inte{y})$ the family of subset of the 2-box $(\inte{x}, \inte{y})$, and let $\diamond$ be one of the four arithmetic operators $+,-,*, \ \text{or} \  /$. Then, the united extension of $f(x,y)=x \diamond y$ is defined as: 
$$\mathcal{F}(\inte{x}, \inte{y}) = \bigcup_{(\inte{x}',  \inte{y}') \in  S(\inte{x}, \inte{y})} \{ x \diamond y : x \in  \inte{x}', \ y \in \inte{y}'\},$$
provided that $0 \notin \inte{y}$ for $\diamond = /$.

\end{definition}

\begin{observation}(United set of a binary operation)\\
Let $Z = \mathcal{F}(\inte{x}, \inte{y})$ be the united set for $f(x,y)=x \diamond y$, and let $\inte{xy} = (\inte{x}, \inte{y})$ be a 2-box. Then, it holds that:
\begin{equation}\label{eq:operational_united_set}
\text{range}(Z) = \{z :  \inf_{x,y \in \inte{xy}} x \diamond y \leq z \leq \sup_{x,y \in \inte{xy}} x \diamond y \} = [ \inf_{x,y \in \inte{xy}} x \diamond y, \ \sup_{x,y \in \inte{xy}} x \diamond y  ].
\end{equation}
\end{observation}

\begin{proof}
The proof follows immediately noticing that $Z \subseteq \mathbb{R}$.
\end{proof}

The range of the united set for all four binary arithmetic operations can be obtained in closed form by means of \eqref{eq:operational_united_set}. Let $\inte{x}=\{x: \lo{x} \leq x \leq \hi{x} \}=[\lo{x}, \hi{x}]$, $\inte{y}=\{y: \lo{y} \leq x \leq \hi{y} \}=[\lo{y}, \hi{y}]$, then equation \eqref{eq:operational_united_set} specialises to:

\begin{align*}
\inte{x} + \inte{y} =\ & [\lo{x} + \lo{y},\ \hi{x} + \hi{y}] \\
\inte{x} - \inte{y} =\ & [\lo{x} - \hi{y},\ \hi{x} - \lo{y}] \\
\inte{x} * \inte{y} =\ & [\min \{ \lo{x}\lo{y},\hi{x}\lo{y},\lo{x}\hi{y},\hi{x}\hi{y} \},\ \max \{ \lo{x}\lo{y},\hi{x}\lo{y},\lo{x}\hi{y},\hi{x}\hi{y} \}] \\
\inte{x} \ / \ \inte{y} =\ & \inte{x} * [1/\hi{y},\  1/\lo{y}], \  \text{for} \  0 \notin \inte{y}\\
\end{align*}

A more efficient form to evaluate the multiplication rule can be found in \cite{neumaier1990interval}.

\subsection{Interval extensions}

With the notion of interval-valued function, and the rules of interval arithmetic we can give a more precise definition of interval computations via the definition of an interval extension.
%a function $f$ can now map an n-interval using the rules of interval arithmetic. 

\begin{definition}(Interval extension)\\
Let $f: \mathbb{R}^n \rightarrow \mathbb{R}^m$, $\inte{x} \in \mathbb{IR}^n$, $x \in \inte{x}$,  and $\mathsf{f}$ be an algebraic expression of $f$. Then, an \emph{interval extension} $\mathsf{F}$ of $f$ via the expression $\mathsf{f}$, is an interval-valued function obtained evaluating the expression $\mathsf{f}$ on $ \inte{x}$ using the rules of interval arithmetic.
\end{definition}

Recall that the algebra of intervals is not distributive, rather it is \emph{sub-distributive}, thus two equivalent algebraic expressions of $f$ may produce different results when evaluated with interval arithmetic rules \cite{neumaier1990interval}. 

\begin{observation}
An interval extension is inclusion monotonic.
\end{observation}

\begin{proof}
Let $\mathsf{F}$ be the interval extension of $f$ via the expression $\mathsf{f}$, whose evaluation consists in applying the rules of interval arithmetic. Note that because $\mathsf{F}$ is a particular instance of an interval-valued function, then to prove that $\mathsf{F}$ is inclusion monotonic, it will be sufficient to show that:
$$\lim_{\inte{x} \rightarrow \hat{x}} \mathsf{F}(\inte{x}) = \mathsf{F}(\hat{x}) \ \ \text{for\ all}  \ \ \hat{x} \in \inte{x}.$$
Since $\mathsf{F}$ and $\mathsf{f}$ are identical algebraic expressions, when evaluated on singletons it holds that $\mathsf{F}(\hat{x}) \equiv \mathsf{f}(\hat{x})$, thus the above limit holds for all $ \hat{x} \in \inte{x}$.
\end{proof}

Note that there are as many interval extensions as there are equivalent algebraic expressions of the function. \\

%Theorem: an interval extension is always inclusion monotonic\\

The following result can be deemed the fundamental theorem of interval computations. The theorem states the renowned conservatism of interval computations. A complete proof of this theorem can be found in the seminal work of Moore \cite{moore1966interval}.

\begin{theorem}(Conservatism of interval computations)\label{th:fundamental}\\
Let $f:\mathbb{R}^n \rightarrow \mathbb{R}^m$ be a function, $\mathcal{F}$ be its united extension, $\mathsf{F}$ be any of its interval extensions, and $\inte{x} \in \mathbb{IR}^n$.  Then, it holds that:
$$\mathcal{F}(\inte{x}) \subset \mathsf{F}(\inte{x}).$$
\end{theorem}

\begin{proof}
Let $Y := \mathcal{F}(\inte{x})$ be the united set of $\mathcal{F}$, and $\inte{y} :=\mathsf{F}(\inte{x})$ be the m-interval of $\mathsf{F}$. Since $\mathsf{F}$ is inclusion monotonic the following must hold:
$$\text{range}(Y) \subset \inte{y}.$$
If this were not true, there would be at least a component of $\inte{y}$ that does not contain the united set, which would imply that there are images in $Y$ that, while complying with the interval constraints defined by $\inte{x}$, are not contained in $\inte{y}$, which in turns would violate the inclusion property for $\mathsf{F}$.
%which proves the original statement.
\end{proof}

Theorem \ref{th:fundamental} in practice states that an interval extension is a so-called outer approximation of its united extension.

\begin{definition}(Optimal extension)\label{def:optimal}\\
Let $f:\mathbb{R}^n \rightarrow \mathbb{R}^m$ be a function, $\mathcal{F}$ be its united extension, $\inte{x} \in \mathbb{IR}^n$, and $Y := \mathcal{F}(\inte{x})$ be the united set of $\mathcal{F}$ applied to $\inte{x}$. Then, a united extension $\mathsf{F}$ is said to be an \emph{optimal extension} if and only if:
$$\text{range}(Y) = \mathsf{F}(\inte{x}).$$
\end{definition}

Note that an optimal extension yields the exact bounds on its corresponding function. 

\begin{observation}(Dependency problem)\label{obs:dependency_problem}\\
Let $\inte{x}, \inte{y} \in \mathbb{IR}$, $\inte{z} = \inte{x} \circ \inte{y}$  and $\inte{w} = \inte{z} \diamond \inte{x}$ be two intervals obtained with the rules of interval arithmetic, and  `$\circ$'  and `$\diamond$' be any two of the four arithmetic operations. Now, let $f(x,y) = \mathsf{f}(x,y)= (x \circ y) \diamond x$, $\inte{w} = \mathsf{F}(\inte{x},\inte{y})$, and $W=\mathcal{F}(x,y)$. Then, it holds that:
$$\text{range}(W) \subseteq \inte{w}.$$
\end{observation}

\begin{proof}
Since there is functional dependence between $\inte{z}$ and $\inte{x}$, the joint set $(\inte{z},\inte{x})$ is a subset of their 2-box. 
Because the rules of interval arithmetic encode no dependence assumption, the second expression $\inte{w}=\inte{z} \diamond \inte{x}$ is evaluated under the assumption that the joint set $(\inte{z},\inte{x})$ is a 2-box, whereas the joint set $(\inte{z},\inte{x})$ is clearly a subset of their 2-box.
%so they assume that the two operands' joint set is indeed a 2-box. 
Thus any further operation between $\inte{z}$ and $\inte{x}$ results in an interval that is wider than it ought to be.
\end{proof}

Because of the dependency problem (Observation \ref{obs:dependency_problem}) of interval extensions, some algebraic expressions evaluate in intervals that are too wide.  Observation   \ref{obs:dependency_problem} goes also by the name of \emph{wrapping effect}.

\begin{observation}(Repeated variables)\label{obs:repeated}\\
Let $f$ be a function, and $\mathsf{f}$ be its algebraic expression with no repeated variables $\inte{x}$. Then its interval extension $\mathsf{F}$ is an optimal extension. 
\end{observation}
%An \emph{optimal extension} is obtained from an algebraic expression that has no repeated variables.\\
\begin{proof}
Without repeated variables in an expression, the dependency problem of Observation \ref{obs:dependency_problem} cannot occur, because the joint dependency structure of all variables is an n-box. Thus, the rules of interval arithmetic encode the correct dependency structure between all variables.
\end{proof}

An example of interval extensions that are not optimal, i.e. carrying repeated variables, are polynomials.
In a typical polynomial expression the number of repeated variables coincides with its degree minus one. So for instance, a second-degree polynomial has one repeated variable. It is generally hard to rearrange polynomial expressions so that no repeated variables occur.

\subsection{Complex intervals}
 
 \begin{definition}(Rectangular complex interval)\\
 Let $\inte{u}, \inte{v} \in \mathbb{IR}$ be two real intervals. A \emph{rectangular complex interval} is a 2-box in the complex plane, and is defined as:
 $$\inte{z}=\inte{u} + \iunit{i} \inte{v}= \{ u + \iunit{i}  v \ | \ u \in \inte{u} , \ v \in \inte{v} \}.$$%\label{eq:complex_interval}
%so just like a complex number is a point in the complex plane whose coordinates are $(u,v)$, a complex interval is a rectangle whose sides are $(\inte{u},\inte{v})$.
 \end{definition}

 \subsubsection{Amplitude}
Let  $z=u+\iunit{i} v \in \mathbb{C}$ be a complex number, with $u,v \in \mathbb{R}$; then its \emph{amplitude} (or absolute value or modulus) is defined as: $|z| = \sqrt{u^2 + v^2}$.  Thus the amplitude of $z$ can be interpreted as the 2-norm of a point in the complex plane whose coordinates are $(u,v)$.
Now let us introduce the amplitude of a rectangular complex number.

\begin{definition}(Amplitude of a complex interval)\\
Let  $\inte{z}=\inte{u}+ \iunit{i} \inte{v} \in \mathbb{IC}$ be a rectangular complex interval, with $\inte{u}, \inte{v} \in \mathbb{IR}$.  The interval amplitude is defined as:
$$|\inte{z}| = \{ \sqrt{u^2 + v^2} : u \in \inte{u}, \ v \in \inte{v} \}.$$
\end{definition}

\begin{observation}(Computing the interval amplitude)\\
Let $|\inte{z}|$ be the amplitude of $\inte{z} \in \mathbb{IC}$, and let $\text{vert}(\inte{z})$ be the set of four vertices of $\inte{z}$, such that:
$\text{vert}(\inte{z}):= \{ \lo{u} + \iunit{i} \lo{v}, \ \hi{u} + \iunit{i} \lo{v}, \ \lo{u} + \iunit{i} \hi{v}, \ \hi{u} + \iunit{i} \hi{v} \}$. Now, let $A$ be the set of amplitudes at each vertex such that:
$$A := \{ |\lo{u} + \iunit{i} \lo{v}|, \ |\hi{u} + \iunit{i} \lo{v}|, \ |\lo{u} + \iunit{i} \hi{v}|, \ |\hi{u} + \iunit{i} \hi{v}| \}.$$
Then, the amplitude of the complex interval $\inte{z}$ is:
\begin{equation*}
 |\inte{z}| = \left\{\begin{matrix}
  [\min A, \max A],\  & \text{if}\ 0 \notin \inte{z},  \\
  [0,\  \max A],\  &  \text{if}\ 0 \in \inte{z}, \\ 
  [\min A_{1} ,\ \max A ],\ & \text{if}\ 0 \in \text{Re}(\inte{z}), \\
  [\min A_{2} ,\ \max A ],\ & \text{if}\ 0 \in \text{Im}(\inte{z}). \\
\end{matrix}\right.
\end{equation*}
Where, $A_1:= \{ |0+ \iunit{i} \lo{v}|, \ |0 + \iunit{i} \hi{v}| \}$, and $A_2 := \{ |\lo{u} + \iunit{i} 0|, \ |\hi{u} + \iunit{i}0| \}$.
\end{observation}

\begin{proof}
In \cite{intervalfourier2020} (Appendix B) a proof for the maximum is provided for any convex set. For the minimum, when $0 \notin \inte{z}$, the proof is equivalent, whilst when $0 \in \inte{z}$, clearly the shortest distance is zero.
\end{proof}
 
 \subsubsection{Argument or phase}
 Let  $z=u + \iunit{i} v \in \mathbb{C}$ be a complex number, with $u,v \in \mathbb{R}$; then its \emph{argument} (or phase), is the angle with the positive real axis, defined as:
 $$\arg z = \arg(u + \iunit{i} v) = \left\{\begin{matrix}
  2 \arctan \left( \frac{v}{\sqrt{u^2 + v^2} + u} \right)  & \text{if}\ u>0,\ v \neq 0, \\
  \pi   &  \text{if}\ u<0,\ v=0, \\ 
  \text{undefined}  & \text{if}\ u=0,\ v=0. \\
\end{matrix}\right.$$.
 
\begin{definition}(Argument of a complex interval)\\
Let  $\inte{z}=\inte{u}+  \iunit{i} \inte{v} \in \mathbb{IC}$ be a rectangular complex interval, with $\inte{u}, \inte{v} \in \mathbb{IR}$.  The interval argument is defined as:
$$\arg \inte{z} = \{ \arg(u + \iunit{i} v) : u \in \inte{u}, \ v \in \inte{v} \}.$$
\end{definition}

\begin{observation}(Computing the interval argument)\label{def:argument}\\
Let $\arg \inte{z}$ be the argument of $\inte{z} \in \mathbb{IC}$, and let $\text{vert}(\inte{z})$ be the set of four vertices of $\inte{z}$, such that:
$\text{vert}(\inte{z}):= \{ \lo{u} + \iunit{i} \lo{v}, \ \hi{u} + \iunit{i} \lo{v}, \ \lo{u} + \iunit{i} \hi{v}, \ \hi{u} + \iunit{i} \hi{v} \}$. Now, let $\Phi$ be the set of arguments of each such vertex such that:
$$\Phi := \{ \arg( \lo{u} + \iunit{i} \lo{v} ), \ \arg( \hi{u} + \iunit{i} \lo{v} ), \ \arg( \lo{u} + \iunit{i} \hi{v} ), \ \arg( \hi{u} + \iunit{i} \hi{v} ) \}.$$
Then the argument of the complex interval $\inte{z}$ is:
\begin{equation*}
 |\inte{z}| = \left\{\begin{matrix}
  [\min \Phi,\ \max \Phi],\  & \text{if}\ 0 \notin \inte{z},  \\
  \text{undefined},\  &  \text{if}\ 0 \in \inte{z}. \\ 
\end{matrix}\right.
\end{equation*}
\end{observation}

\subsubsection{Arithmetic between complex intervals}
In order to do interval computations between complex intervals the rules of arithmetic must be defined. For the purpose of this work we are only going to need the sum and subtraction between complex intervals, as well as the multiplication of an interval times a complex number.

\begin{definition}(Sum and subtraction)\\
Let $\inte{z}, \inte{w} \in \mathbb{IC}$ be two complex intervals, $\text{vert}(\inte{z})$ and $ \text{vert}(\inte{w})$ be their set of four vertices, and $\diamond = \{+,-\}$ be the sum and subtraction operators. Now let $\text{vert}(\inte{z}\inte{w}) = \text{vert}(\inte{z}) \times \text{vert}(\inte{w})$ be the Cartesian product of all pairs of vertices. Then the arithmetic operation specialises to:
$$\inte{z} \diamond \inte{w} = \{z \diamond w: (z,w) \in \text{vert}(\inte{z}\inte{w}) \}.$$
\end{definition}

Note that the cardinality of the product set is  $ \# \{ \text{vert}(\inte{z}\inte{w}) \} = 16$, so exactly $16$ sums/subtractions need to be computed. More efficient implementations of such operations can be derived noticing that the sum and subtractions only takes place between diagonally opposite pairs of vertices.

\begin{definition}(Multiplication between an interval and a complex number)\\
Let $\inte{x} \in \mathbb{IR}$ be an interval, $z :=u + \iunit{i} v$ be a complex number, and $ \{ u \lo{x}, u \hi{x} \}$ be the set of possible real multiplications. Then, the multiplication between an interval and a complex number is given by:
$$\inte{x} * z = z * \inte{x} =  [\min \{ u \lo{x}, u \hi{x} \} + \iunit{i} v,\  \max \{ u \lo{x}, u \hi{x} \} + \iunit{i} v].$$
\end{definition}

\newpage
\section{Extensions of the discrete Fourier transform}\label{sec:extensions}

\subsection{Interval extension}
Recall that an interval extension is any algebraic expression, whose entries have been replaces with intervals, and that an interval extension is always inclusion monotonic. In this section we define the interval extension considered throughout this manuscript.

\begin{definition}(DFT interval extension)\label{def:interval_dft}\\
Let $\inte{x} \in \mathbb{IR}^N$ be a sequence of real intervals, and let $\mathsf{f}_h: \mathbb{R}^N \rightarrow \mathbb{C}$ be the expression \eqref{eq:algebraic_dft}. Then, the \emph{interval extension} of the discrete Fourier transform at harmonic $ h=0,...,N-1$, $\mathsf{F}_h: \mathbb{IR}^N \rightarrow \mathbb{IC}$ is simply given by: 

\begin{equation}\label{eq:interval_dft}
\mathsf{F}_h(\inte{x}) := \sum_{n=0}^{N-1} \inte{x}_n \  e^{-\textcolor{red}{i} \frac{2 \pi}{N} h n}
\end{equation}
\noindent
where, $N$ is the size of the N-interval (usually a power of two). The evaluation of such an expression results in a rectangular complex interval denoted by: $\inte{z}_h := \mathsf{F}_h(\inte{x})$.
\end{definition}

Note that \eqref{eq:interval_dft} is an optimal interval extension, under no dependency statement between the components of $\inte{x}$. In other words, the interval extension \eqref{eq:interval_dft} is optimal because it has no repeated variables in its expression.

\subsection{United extension and united set} 
\begin{definition}(United extension and united set)\\
The \emph{united extension} of the discrete Fourier transform on the N-interval $\inte{x} \in \mathbb{IR}^N$, is:
\begin{equation}\label{eq:united_dft}
\mathcal{F}_h  (\inte{x})  := \bigcup_{\inte{x}' \in S(\inte{x})}  \left\{ \sum_{n=0}^{N-1} x_n \  e^{-\textcolor{red}{i} \frac{2 \pi}{N} h n}: x \in \inte{x}' \right\}.
\end{equation}
 \noindent
 The image of \eqref{eq:united_dft} is named \emph{united set}, and will be denoted by $ Z_h:=  \mathcal{F}_h (\inte{x}) $, with $Z_h \subset \mathbb{C}$. %Let $\boldsymbol{x}$ and $\boldsymbol{y}$, be two nested intervals, with $\boldsymbol{x} \subseteq \boldsymbol{y}$. 
 \end{definition}

 Note that the united set $Z_h$ is not a box in $\mathbb{C}$; this is why it is not trivial to obtain the exact bounds on the amplitude and phase of the discrete Fourier transform.
Observe that because of the property of inclusion monotonicity of interval extensions, and because the extension defined in \ref{def:interval_dft} is an optimal extension it holds that:
 \begin{equation}\label{eq:inclusion}
 Z_h  \subseteq \inte{z}_h \ \ \ \Rightarrow \ \ \ \text{range}(Z_h) = \inte{z}_h,
 \end{equation}
 
 \noindent
thus the box $\boldsymbol{z}_h$ always circumscribes the united set $Z_h$. In other words, taking the range of the united set $Z_h$ is equivalent to evaluating the interval extension $\mathsf{F}_h$.

 \section{Reaching the united set}\label{sec:reach}
 
\begin{definition}(Non-interactive intervals or boxes)\label{def:noninteractive}\\
Let $\inte{x}, \inte{y} \in \mathbb{IR}$. Then, $\inte{x}$ and $\inte{y}$ are said to be \emph{non-interactive} if their joint set is a 2-box, i.e. if it holds that: $(\inte{x}, \inte{y}) :=  \{ (x,y) : x \in \inte{x}, \ y \in \inte{y} \} $. 
\end{definition}

The marginal projections $\inte{x}$ and $\inte{y}$, retain all the information about the joint structure. 
Also, when no statement about the dependence between $\inte{x}$ and $\inte{y}$, the joint structure of Definition \ref{def:noninteractive} is implied. In other terms, when nothing is stated about the dependence between two intervals, these will be considered non-interactive.

\begin{remark}(Non-interactive independence)\\
Two non-interactive intervals are often referred to as independent in the literature.  Whilst non-interactivity is a form of interval independence, we tend to avoid the term independence because of the probabilistic interpretation of intervals, which can be seen as sets of bounded probability distributions. Under the probabilistic interpretation saying that two intervals are independent may imply statistical independence, which is a much stronger statement than non-interactivity.  If no dependency statement is made, the components of an n-interval are mutually non-interactive.
\end{remark}

\begin{definition}(Linearly dependent intervals)\label{def:linear_dep}\\
Let $\inte{x}, \inte{y} \in \mathbb{IR}$. Then $\inte{x}$ and $\inte{y}$ are said to be \emph{linearly dependent} if their joint set is defined as:
$$(\inte{x},\inte{y})_{\bigtimes}  := \{ (x,y) : x=\text{mid} \inte{x} \pm \tau \ \text{rad}  \inte{x}, \ y=\text{mid} \inte{y} \pm \tau \ \text{rad}  \inte{y},\ \tau \in [-1,1] \},$$
\noindent
where we recall that $\text{mid}\inte{x} = (\hi{x}+\lo{x})/2$, and $\text{rad}\inte{x} =(\hi{x}-\lo{x})/2$. Note that choosing a pair of signs $\pm=\{ +,- \}$,  determines the orientation of the joint set.
\end{definition}

\begin{definition}(Interval diagonal)\\
The joint set of two linearly dependent intervals, denoted by $(\inte{x},\inte{y})_{\bigtimes}$, is called an \emph{interval diagonal}. An interval diagonal can be regarded as the simplest convex polytope in $\mathbb{R}^2$, with only two vertices. The vertices of such polytope can be obtained by setting $\tau=\{-1,1 \}$:
$$\text{vert}(\inte{x},\inte{y})= \{ (\inte{x},\inte{y})_{\bigtimes} : \tau= \{ -1,1 \} \}.$$
\end{definition}

\begin{definition}(Oriented linearly dependent intervals)\label{def:oldi}
Let $\inte{x}, \inte{y} \in \mathbb{IR}$ be two linearly dependent intervals, and $t \in [0,1]$. An orientation can be established within their joint set $(\inte{x},\inte{y})_{\bigtimes}$ simply by ordering each pair $(x,y)$ by increasing $t$. Four orientations can be established for linearly dependent intervals:\\

\begin{tabular}{lc}
Perfect ($+$)     & $(\inte{x},\inte{y})_{\nearrow} =\{ (x,y) : x=\lo{x} + t (\hi{x}-\lo{x}), \ y=\lo{y} + t (\hi{y}-\lo{y}),\ t \in [0,1] \}$, \\
Perfect ($-$)    & $(\inte{x},\inte{y})_{\swarrow}=\{ (x,y) : x=\hi{x} - t (\hi{x}-\lo{x}), \ y=\hi{y} - t (\hi{y}-\lo{y}),\ t \in [0,1] \}$, \\
Opposite ($+$)  & $(\inte{x},\inte{y})_{\nwarrow}= \{ (x,y) : x=\hi{x} - t (\hi{x}-\lo{x}), \ y=\lo{y} + t (\hi{y}-\lo{y}),\ t \in [0,1] \}$, \\
Opposite ($-$) & $(\inte{x},\inte{y})_{\searrow}= \{ (x,y) : x=\lo{x} + t (\hi{x}-\lo{x}), \ y=\hi{y} - t (\hi{y}-\lo{y}),\ t \in [0,1] \}$. \\
\end{tabular}
\end{definition}

\begin{definition}(Oriented interval diagonal)\\
The joint set of two oriented linearly dependent intervals can be given the interpretation of an \emph{oriented interval diagonal}, i.e. a convex set with two ordered vertices. Given Definition \ref{def:oldi} first and second vertex of an oriented diagonal can be obtained setting $t=0$ and $t=1$ respectively.

\end{definition}

\begin{definition}(Linear interval dependence)\\
The joint set $(\inte{x},\inte{y})_{\bigtimes}$ of Definition \ref{def:linear_dep} expresses linear dependence between intervals, a.k.a. \emph{linear interval dependence}. 
\end{definition}

\begin{definition}(Strong linear interval dependence)\label{def:strong_lin_dep}\\
Let $\inte{x}_1, \inte{x}_2 \in \mathbb{IR}$, and let $u, v \in  \mathbb{R}$. The two intervals $\inte{x}_1$ and $\inte{x}_2$ are said to have \emph{strong linear interval dependence}, if there exists a generator $\inte{x} \in \mathbb{IR}$, such that:
$$(\inte{x}_1,\inte{x}_2) = (u \inte{x} , v \inte{x}) = \{ (x_1,x_2) : x_1 = u x,\ x_2=v x,\ x \in \inte{x} ,\ u,v \in \mathbb{R}  \}.$$
Because the joint set $(u \inte{x} , v \inte{x})$ is determined by its generator $\inte{x}$, such joint set can be denoted by: $\ _{uv}\inte{x} := (u \inte{x} , v \inte{x})$. The vertices of the diagonal joint set are given by: $ \text{vert} ( _{uv} \inte{x}) = \{ ( u \lo{x}, v \lo{x} ), ( u \hi{x}, v \hi{x} ) \} $. 
\end{definition}

\begin{remark}
Note that because strong linear dependence is a more stringent condition than linear dependence, any two linearly dependent interval may not satisfy Definition \ref{def:strong_lin_dep}.
\end{remark}

\begin{remark}
Observe that the orientation of the interval diagonal induced by strong linear dependence is fully determined by the sign of $u$ and $v$. This holds under the classical interpretation that the interval $\inte{x} \in \mathbb{IR}$ is oriented like the real line $\mathbb{R}$. Improper interval analysis and Kraucher arithmetic have challenged this interpretation \cite{kaucher1980interval}.
\end{remark}

\begin{example}(Addend of the Fourier series)\\
The addend of the Fourier series is an example of an oriented interval diagonal in the complex plane.  The Fourier coefficients are $w_{hn}=u_{hn}+iv_{hn}=u_{n}+iv_{n}$.
Let $h=1$, $N=8$, then $w_{1,0} = e^{0}=1+i0$ with $u_0=1,\ v_0=0$; $w_{1,1} = e^{-i \frac{2\pi}{8}}=\cos \frac{\pi}{4}  - i \sin \frac{\pi}{4}$, with $u_1 = \cos \frac{\pi}{4}, \ v_1= \sin \frac{\pi}{4}$; $w_{1,2} = e^{-i \frac{2\pi}{8} 2}=\cos \frac{\pi}{2}  - i \sin \frac{\pi}{2}$, with $u_2 = 0, \ v_2= 1$.

\vspace{1em}
\begin{tikzpicture}

\node[] (xx0) at (-3.5,5) {$\inte{x}_{0} = [-2,0] $}; 
\node[] (X0r) at (-3.5,4) {$\inte{x}_{0} \cdot u_0 $}; 
\node[] (X0r) at (-0.9,3.75) {$ + $}; 

\node[] (xx1) at (1.3,5) {$\inte{x}_{1} = [1,3] $}; 
\node[] (X0r) at (1.3,4) {$\inte{x}_{1} \cdot u_1 $}; 
\node[] (X0r) at (3.3,3.75) {$ + $}; 

\node[] (xx2) at (6,5) {$\inte{x}_{2} = [-4,-2] $}; 
\node[] (X0r) at (6,4) {$\inte{x}_{2} \cdot u_2 $}; 
\node[] (X0r) at (8.3,3.75) {$ + \ \ ...$};

\node[] (i) at (-4.5,3.5) {$\textcolor{red}{i}$}; 
\node[] (X0r) at (-3.5,3.5) {$\inte{x}_{0} \cdot v_0 $}; 
\node[] (i) at (0.3,3.5) {$\textcolor{red}{i}$}; 
\node[] (X0r) at (1.3,3.5) {$\inte{x}_{1} \cdot v_1 $}; 
\node[] (i) at (5,3.5) {$\textcolor{red}{i}$}; 
\node[] (X0r) at (6,3.5) {$\inte{x}_{2} \cdot v_2 $};

\begin{axis}[name=axis1,width=4cm, height=4cm, at={(-.3\linewidth,0)},
  axis lines=middle,
  grid=major,
  xmin=-4,
  xmax=2,
  ymin=-2,
  ymax=4,
  xlabel=Re,
  ylabel=Im,
  xtick={-4,-3,...,2},
  ytick={-2,-1,...,4},
  ticklabel style = {font=\tiny},
  x label style = {font=\tiny},
  y label style = {font=\tiny},
]

% \addplot[-{Latex[length=2mm, width=2mm]}, blue, very thick] coordinates {(-2,0) (0,0)};
\addplot[{Circle[length=1.mm, width=1.mm]}-{Diamond[length=1.5mm, width=1.5mm]}, blue, very thick] coordinates {(-2,0) (0,0)};
\end{axis}  

\node[] (+) at (-0.9,1.2) {{\Large +}};   

\begin{axis}[name=axis2,width=4cm, height=4cm,at={(0\linewidth,0)},
  axis lines=middle,
  grid=major,
  xmin=-4,
  xmax=2,
  ymin=-2,
  ymax=4,
  xlabel=Re,
  ylabel=Im,
  xtick={-4,-3,...,2},
  ytick={-2,-1,...,4},
  ticklabel style = {font=\tiny},
  x label style = {font=\tiny},
  y label style = {font=\tiny},
]

% \addplot[-{Latex[length=2mm, width=2mm]}, blue, very thick] coordinates {(-2.12132,0.707107) (-0.707107,2.12132)};
\addplot[{Circle[length=1.mm, width=1.mm]}-{Diamond[length=1.5mm, width=1.5mm]}, blue, very thick] coordinates {(-2.12132,0.707107) (-0.707107,2.12132)};

\end{axis} 

\node[] (+) at (3.3,1.2) {{\Large +}};

\begin{axis}[name=axis3,width=4cm, height=4cm,at={(.3\linewidth,0)},
  axis lines=middle,
  grid=major,
  xmin=-4,
  xmax=2,
  ymin=-2,
  ymax=4,
  xlabel=Re,
  ylabel=Im,
  xtick={-4,-3,...,2},
  ytick={-2,-1,...,4},
  ticklabel style = {font=\tiny},
  x label style = {font=\tiny},
  y label style = {font=\tiny},
]

% \addplot[-{Latex[length=2mm, width=2mm]}, blue, very thick] coordinates {(0,2) (0,4)};
\addplot[{Circle[length=1.mm, width=1.mm]}-{Diamond[length=1.5mm, width=1.5mm]}, blue, very thick] coordinates {(0,2) (0,4)};
\end{axis}

\node[] (+) at (8.3,1.2) {{\Large +} \ \  ...};  
\end{tikzpicture}
\end{example}

\begin{definition}(Minkowski addition)\label{def:minkowski}\\
Let $A, B \in \mathcal{C}$ be any two convex subsets of a real vector space, where $\mathcal{C} \subseteq \mathbb{R}^m$ is such a space of convex subsets. Then the \emph{Minkowski addition} is given by:
$$A \oplus B = \bigcup_{A,B \in \mathcal{C}} \{ a+b : a \in A, b \in B \}.$$
\end{definition}

The following is a well-known result in computational geometry, which leads to a practical implementation of the Minkowski sum. 

\begin{proposition}\label{prop:minkowski}
The sum of two convex polytopes in the plane: $A, B \in \mathcal{C} \subseteq \mathbb{R}^2$ with $m, n$ vertices, is a convex polytope with at most $m+n$ vertices. 
\end{proposition}

\begin{remark}
Note that because of Proposition \ref{prop:minkowski} the coordinates of the vertices resulting from the Minkowski addition can be computed in linear time $O(m+n)$. The intuitive procedure is provided in the form of text in the Wikipedia article \cite{wiki:minkowski_addition}.
\end{remark}

%the Wikipedia page of the Minkowski addition: \url{https://en.wikipedia.org/wiki/Minkowski_addition}, with deep link: \url{#Algorithms_for_computing_Minkowski_sums}. 

Here is the extract of the Wikipedia article on December 2021: \\
\textit{For two convex polygons $P$ and $Q$ in the plane with $m$ and $n$ vertices, their Minkowski sum is a convex polygon with at most $m + n$ vertices and may be computed in time $O(m + n)$ by a very simple procedure, which may be informally described as follows. Assume that the edges of a polygon are given and the direction, say, counterclockwise, along the polygon boundary. Then it is easily seen that these edges of the convex polygon are ordered by polar angle. Let us merge the ordered sequences of the directed edges from $P$ and $Q$ into a single ordered sequence $S$. Imagine that these edges are solid arrows which can be moved freely while keeping them parallel to their original direction. Assemble these arrows in the order of the sequence $S$ by attaching the tail of the next arrow to the head of the previous arrow. It turns out that the resulting polygonal chain will in fact be a convex polygon which is the Minkowski sum of $P$ and $Q$.}
%\textit{If one polygon is convex and another one is not, the complexity of their Minkowski sum is $O(nm)$. If both of them are nonconvex, their Minkowski sum complexity is $O((mn)^2)$.}

\begin{remark}
Note that as a consequence of Proposition \ref{prop:minkowski}, the Minkowski addition is convexity preserving.
\end{remark}

\begin{proposition}(Minkowski addition and united extensions)\label{prop:minkowski2}\\
Let $\mathsf{f}: \mathbb{R}^2 \rightarrow \mathbb{R}$ be the sum function, such that $\mathsf{f}(x,y)=x+y$, and $\inte{x}, \inte{y} \in \mathbb{IR}$ be two dependent intervals, whose joint set is convex, such that $(\inte{x},\inte{y}) \in \mathcal{C}$. Then, the united extension of $\mathsf{f}$, $\mathcal{F}_{+}(\inte{x},\inte{y})$ is equivalent to the Minkowski sum, such that:
$$\mathcal{F}_{+}(\inte{x},\inte{y}) \equiv \inte{x} \oplus \inte{y},$$
thus the united set of $\mathcal{F}_{+}$ can be reached by means of the Minkowski sum. 
%The united extension of the sum function coincides with the Minkowski sum.
\end{proposition}

\begin{proof}
Invoking the Definition \ref{def:united_extension}, the united extension of the sum function applied to an arbitrary set $(\inte{x}, \inte{y})$ is given by:
%$$\mathcal{F}_{+}(\inte{x},\inte{y}) = \bigcup_{(\inte{x},\inte{y}) \in \mathcal{C}} \{ x+y: (x,y) \in (\inte{x}, \inte{y})  \},$$
$$\mathcal{F}_{+}(\inte{x},\inte{y}) = \bigcup_{\inte{x}', \inte{y}' \in S(\inte{x},\inte{y})} \{ x+y: x \in \inte{x}', \ y \in \inte{y}'  \},$$
which is equivalent to Definition \ref{def:minkowski}, when $(\inte{x},\inte{y})$ is a convex subset of $\mathbb{R}^2$. Therefore the Minkowski sum coincides to evaluating  the united extension of the sum function. 
\end{proof}

\begin{corollary}
The Minkowski sum of two linearly dependent intervals $\inte{x}, \inte{y} \in \mathbb{IR}$, whose joint set is: $(\inte{x},\inte{y}) = (\inte{x},\inte{y})_{\bigtimes}$, yields the united set under the sum function.
\end{corollary}

\begin{proof}
The proof is identical to the proof of Proposition \ref{prop:minkowski2}, and it holds because the joint set of two linearly dependent intervals is convex: $(\inte{x},\inte{y})_{\bigtimes} \in \mathcal{C}$.
\end{proof}

\begin{example}(Minkowski addition between two diagonals)\\
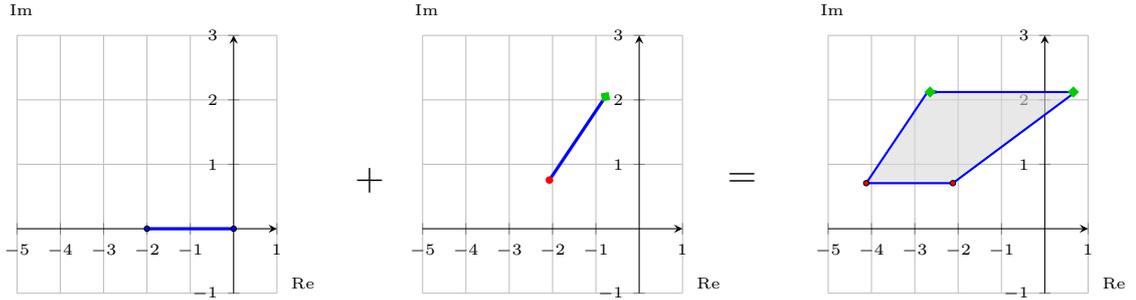
\begin{tikzpicture}
\begin{axis}[name=axis1,width=5cm, height=5cm, at={(-.35\linewidth,0)},
  axis lines=middle,
 grid=major,
  xmin=-5,
  xmax=1,
  ymin=-1,
  ymax=3,
  xlabel=Re,
  ylabel=Im,
  xtick={-5,-4,...,1},
  ytick={-1,0,...,3},
  ticklabel style = {font=\tiny},
  x label style = {font=\tiny, at={(axis description cs:1.1,0.1)},anchor=north},
  y label style = {font=\tiny, at={(axis description cs:0.1,1.1)},anchor=east},
]

% \addplot[-{Latex[length=2mm, width=2mm]}, blue, very thick] coordinates {(-2,0) (0,0)};
\addplot[blue, very thick] coordinates {(-2,0) (0,0)};
\addplot [fill=blue] (-2,0) circle (1pt);
\addplot [fill=blue] (0,0) circle (1pt);
\end{axis} 

\node[] (+) at (-0.7,1.5) {{\Large +}};

\begin{axis}[name=axis2,width=5cm, height=5cm,at={(0\linewidth,0)},
  axis lines=middle,
 grid=major,
  xmin=-5,
  xmax=1,
  ymin=-1,
  ymax=3,
  xlabel=Re,
  ylabel=Im,
  xtick={-5,-4,...,1},
  ytick={-1,0,...,3},
  ticklabel style = {font=\tiny},
  x label style = {font=\tiny, at={(axis description cs:1.1,0.1)},anchor=north},
  y label style = {font=\tiny, at={(axis description cs:0.1,1.1)},anchor=east},
]

% \addplot[-{Latex[length=2mm, width=2mm]}, blue, very thick] coordinates {(-2.12132,0.707107) (-0.707107,2.12132)};
\addplot[{Circle[length=1.mm, width=1.mm, color=red]}-{Diamond[length=1.5mm, width=1.5mm, color=green!80!black]}, blue, very thick] coordinates {(-2.12132,0.707107) (-0.707107,2.12132)};
% \addplot [red, fill=red] (-2.12132,0.707107) circle (1pt);
% \addplot [green, fill=green] (-0.707107,2.12132) circle (1pt);
\end{axis}

\node[] (=) at (4.2,1.5) {{\Large =}};

\begin{axis}[name=axis3,width=5cm, height=5cm,at={(.35\linewidth,0)},
  axis lines=middle,
  grid=major,
  xmin=-5,
  xmax=1,
  ymin=-1,
  ymax=3,
  xlabel=Re,
  ylabel=Im,
  xtick={-5,-4,...,1},
  ytick={-1,0,...,3},
  ticklabel style = {font=\tiny},
  x label style = {font=\tiny, at={(axis description cs:1.1,0.1)},anchor=north},
  y label style = {font=\tiny, at={(axis description cs:0.1,1.1)},anchor=east},
]

\addplot[blue, thick,fill=gray!30,fill opacity=0.6] (0.707107, 2.12132) -- (-2.707107, 2.12132) -- (-4.12132, 0.707107)  --  (-2.12132, 0.707107)    --  cycle;
% \addplot [fill=green] (0.707107, 2.12132) circle (1pt);
% \addplot [fill=green] (-2.707107, 2.12132) circle (1pt);
\addplot[{Diamond[length=1.5mm, width=1.5mm, color=green!80!black]}-, blue, very thick] coordinates {(-2.78, 2.12132) (-2.5, 2.12132)};
\addplot[{Diamond[length=1.5mm, width=1.5mm, color=green!80!black]}-, blue, very thick] coordinates {(0.79, 2.12132) (0.68, 2.12132)};
\addplot [fill=red] (-4.12132, 0.707107) circle (1pt);
\addplot [fill=red] (-2.12132, 0.707107) circle (1pt);
\end{axis}

\end{tikzpicture} 
\end{example}

\begin{example}(Minkowski addition between a convex set and a diagonal)\\
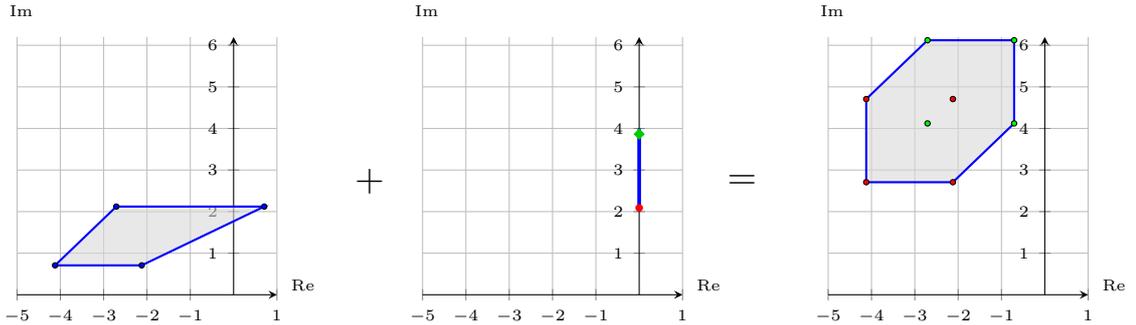
\begin{tikzpicture}

\begin{axis}[name=axis3,width=5cm, height=5cm,at={(-.35\linewidth,0)},
  axis lines=middle,
  grid=major,
  xmin=-5,
  xmax=1,
  ymin=0,
  ymax=6.2,
  xlabel=Re,
  ylabel=Im,
  xtick={-5,-4,...,1},
  ytick={0,1,...,6},
  ticklabel style = {font=\tiny},
  x label style = {font=\tiny, at={(axis description cs:1.1,0.1)},anchor=north},
  y label style = {font=\tiny, at={(axis description cs:0.1,1.1)},anchor=east},
]

\addplot[blue, thick,fill=gray!30,fill opacity=0.6] (0.707107, 2.12132) -- (-2.707107, 2.12132) -- (-4.12132, 0.707107)  --  (-2.12132, 0.707107)    --  cycle;
\addplot [fill=blue] (0.707107, 2.12132) circle (1pt);
\addplot [fill=blue] (-2.707107, 2.12132) circle (1pt);
\addplot [fill=blue] (-4.12132, 0.707107) circle (1pt);
\addplot [fill=blue] (-2.12132, 0.707107) circle (1pt);
\end{axis}

\node[] (+) at (-0.7,1.5) {{\Large +}};  

\begin{axis}[name=axis3,width=5cm, height=5cm,at={(0\linewidth,0)},
  axis lines=middle,
  grid=major,
  xmin=-5,
  xmax=1,
  ymin=0,
  ymax=6.2,
  xlabel=Re,
  ylabel=Im,
  xtick={-5,-4,...,1},
  ytick={0,1,...,6},
  ticklabel style = {font=\tiny},
  x label style = {font=\tiny, at={(axis description cs:1.1,0.1)},anchor=north},
  y label style = {font=\tiny, at={(axis description cs:0.1,1.1)},anchor=east},
]

\addplot[blue, very thick] coordinates {(0,2) (0,4)};
% \addplot+[mark=square, style={solid, fill=gray}, mark size=1pt, mark options={solid}] (0, 2);
% \addplot+[only marks,mark=square, mark options={scale=1, solid}, text mark as node=true] coordinates {(0,2)};
% \addplot [red, mark=square, mark size=1pt] coordinates {(0,2)};
% \addplot [thin, black, fill=green!90!black] (0, 4) circle (1pt);
\addplot[{Circle[length=1.mm, width=1.mm, color=red]}-{Diamond[length=1.5mm, width=1.5mm, color=green!80!black]}, blue, very thick] coordinates {(0,2) (0,4)};
\end{axis}

\node[] (=) at (4.2,1.5) {{\Large =}};

\begin{axis}[name=axis3,width=5cm, height=5cm,at={(.35\linewidth,0)},
  axis lines=middle,
grid=major,
  xmin=-5,
  xmax=1,
  ymin=0,
  ymax=6.2,
  xlabel=Re,
  ylabel=Im,
  xtick={-5,-4,...,1},
  ytick={0,1,...,6},
  ticklabel style = {font=\tiny},
  x label style = {font=\tiny, at={(axis description cs:1.1,0.1)},anchor=north},
  y label style = {font=\tiny, at={(axis description cs:0.1,1.1)},anchor=east},
]

\addplot[blue, thick,fill=gray!30,fill opacity=0.6] (-0.707107, 6.12132) -- (-2.707107, 6.12132) -- (-4.12132, 4.707107) --(-4.12132, 2.707107) -- (-2.12132, 2.707107) -- (-0.707107, 4.12132)   --  cycle;
\addplot [fill=green] (-0.707107, 6.12132) circle (1pt);
\addplot [fill=green] (-2.707107, 6.12132) circle (1pt);
\addplot [fill=red]  (-4.12132, 4.707107) circle (1pt);
%\addplot [red, mark=square, mark size=1pt] coordinates {(-4.12132, 4.707107)};
\addplot [fill=red] (-4.12132, 2.707107) circle (1pt);
%\addplot [red, mark=square, mark size=1pt] coordinates {(-4.12132, 2.707107)};
\addplot [fill=red] (-2.12132, 2.707107) circle (1pt);
%\addplot [red, mark=square, mark size=1pt] coordinates {(-2.12132, 2.707107)};
\addplot [fill=green] (-0.707107, 4.12132) circle (1pt);
\addplot [fill=red] (-2.12132, 4.707107) circle (1pt);
%\addplot [red, mark=square, mark size=1pt] coordinates {(-2.12132, 4.707107)};
\addplot [fill=green] (-2.707107, 4.12132) circle (1pt);

% \addplot[{Circle[length=1.mm, width=1.mm, color=red]}-{Diamond[length=1.5mm, width=1.5mm, color=green!80!black]}, blue, very thick] coordinates {(-2.12132,0.707107) (-0.707107,2.12132)};

\end{axis}
\end{tikzpicture} 
\end{example}

We are ready to state the main result of this work.

 \begin{theorem}(Main result)\label{theor:main} \\
Let $\inte{x} \in \mathbb{IR}^N$ be an n-interval, whose components are non-interactive (N-box). Let $\mathsf{F}_h$ be the discrete 
Fourier transform at a given harmonic $h = 0,1,...,N-1$, and let $Z_h$ be the united set, obtained applying the united extension $\mathcal{F}_h$ on to $\inte{x}$. Then, the united set $Z_h$ is a convex polytope in $\mathbb{C}$. Thus, its vertices are computable by a chain of Minkowski additions.%, with at most $2N$ vertices. 
\end{theorem}

 \begin{proof}
 Let us expand the Fourier transform and consider its algebraically equivalent trigonometric expression: 
 $$\mathsf{f}_h(x) = \sum_{n=0}^{N-1} x_n \ \left(  \cos{\frac{2 \pi}{N} h n} - \textcolor{red}{i} \sin{\frac{2 \pi}{N} h n} \right).$$ 
 \noindent
 With the change of variables:
 $$u_{hn} = \cos{\frac{2 \pi}{N} h n}, \ \ \ \ v_{hn} = -\sin{\frac{2 \pi}{N} h n},$$
 \noindent
the transform becomes:
\begin{equation}\label{eq:expression2}
\mathsf{f}_h(x) = \sum_{n=0}^{N-1} x_n  u_{hn} + \textcolor{red}{i}\ x_n v_{h n}.
\end{equation}
\noindent
Let us now consider the interval extension of \eqref{eq:expression2}: $\mathsf{F}_h(\inte{x}) = \sum_{n=0}^{N-1} \inte{x}_n  u_{hn} + \textcolor{red}{i}\ \inte{x}_n v_{h n}$.  Each addend of such extension is the interval diagonal $_{uv}\inte{x}_{hn} \in \mathbb{C}$, given by: 
$$_{uv}\inte{x}_{hn} := \inte{x}_n  u_{hn} + \textcolor{red}{i}\ \inte{x}_n v_{h n}.$$
 \noindent
Each interval diagonal  $_{uv}\inte{x}_{hn},\ n=0,1,...,N-1$ is a convex polytope in $\mathbb{C}$ with two vertices.
 \noindent
 Substituting the interval diagonals $_{uv}\inte{x}_{hn}$ back into expression \eqref{eq:expression2}, the interval extension of the discrete Fourier transform becomes:
 \begin{align*}
 \mathsf{F}_h(\inte{x}) & = \sum_{n=0}^{N-1}\ _{uv}\inte{x}_{hn}\\
 									   & = (\cdots((_{uv}\inte{x}_{h0} \ +\ _{uv}\inte{x}_{h1}) \ +\ _{uv}\inte{x}_{h2})  \ +...+\ _{uv}\inte{x}_{h(N-1)}).
 \end{align*}
 \noindent
 Now, because each interval diagonal is a convex subset of $\mathbb{C}$, we can replace each addition by its Minkowski equivalent:
 \begin{equation}\label{eq:minkowski_series}
 Z_h = \mathcal{F}_h(\inte{x}) = (\cdots((_{uv}\inte{x}_{h0} \ \oplus \ _{uv}\inte{x}_{h1}) \  \oplus \ _{uv}\inte{x}_{h2})  \  \oplus ...  \oplus  \ _{uv}\inte{x}_{h(N-1)}).
 \end{equation}
 \noindent
Because the sum between two diagonals is a convex set, and the sum of a convex set and a diagonal is again a convex set, by invoking Proposition \ref{prop:minkowski}, it holds that the series of nested Minkowski sums  of \eqref{eq:minkowski_series} yields the united set $Z_h$, which concludes the proof.
% Thus the united set $Z_h$ can be computed as a series of chained Minkowski additions.
 %whose number of vertices is at most $2N$, which concludes the proof. 
 \end{proof}

 \begin{remark}
 Note that the proof of Theorem \ref{theor:main} implies that the dependence encoded in the diagonal is fully preserved, and so the dependency problem of interval computations is fully addressed.  The  resulting united set $Z_h$ is therefore a convex polytope in $\mathbb{C}$, whose vertices can be computed by repeatedly applying the Minkowski sum in \eqref{eq:minkowski_series}.
 \end{remark}
 
 \begin{corollary}(Number of vertices)\\
 Let $Z_h$ be he united set of \eqref{eq:minkowski_series}, then $Z_h$ is convex and  has at most $2N$ vertices.
 \end{corollary}
 
 \begin{proof}
 The proof follows from Proposition \ref{prop:minkowski}, noticing that each sum in \eqref{eq:minkowski_series} is a Minkowski addition between convex sets.
 \end{proof}

 \section{Discussion}\label{sec:discussion}
This paper has set out the mathematical foundations upon which a polynomial-time algorithm can be built, in order to compute exact bounds on the amplitude and phase of the discrete Fourier transform, when presented with input uncertainty in the form of intervals. 
 
% First, some elements of interval computations are introduced to 

The main result states that the united set $Z_h \subset \mathbb{C}$ of the transform can be reached by chaining Minkowski additions. This provides a means to address the dependency problem of interval computations, whilst characterising the geometry of the united set. This result also establishes the convexity of such sets, which in turns makes it possible to obtain exact bounds on the amplitude and phase.  

Note from Definition \ref{def:argument} that when the united set at a particular harmonic contains the origin of the complex plane, the phase for such a harmonic is undefined. 

With a full characterization of the united set it is possible to pinpoint the configuration of endpoints in the original N-interval that leads to the bounds for each harmonic. There will be one configuration for the upper bound and one for the lower bound for each harmonic. This study can help the analysts identify some `critical' signals within the N-interval that lead to the upper bound of the amplitude, which can be useful when concerned with the energy content of the process. 
%nverse Fourier transform \\

It is important to note that an algorithm for the exact bounds on amplitude and phase can suffer additional dependency problems if used to perform further calculations involving the interval amplitude and/or the interval phase.  Therefore, such an algorithm must be used with caution, while understanding the limitations of interval computations.

%Further dependency problems. \\

Proposition \ref{prop:minkowski} induces an algorithm that can compute the vertices of a Minkowski addition in linear time. This is a very interesting prospect as it can lead to faster algorithms than the one presented in \cite{intervalfourier2020}. In fact, with a linear-time algorithm for the Minkowski addition, there is no longer need to compute a convex hull at each addition, bringing significantly down the cost of the algorithm.
%A faster algorithm.\\

Finally, a short remark about verified computations. In order for the computations presented in this paper to be verified and exact at the same time, additional theoretical work is needed. Verified computations will need to deal with interval Fourier coefficients, whose width is typically orders of magnitude smaller compared to the width of the input uncertainty. Because of the difference in magnitude, the exactness of the bounds could be relaxed to make room for rigour, however, to achieve verified computations a slightly modified implementation of the algorithm is needed.
%Verified computations. \\
 
 \section*{Acknowledgements}

Thanks to  Tania Gricel Benitez, Ander Gray, Scott Ferson, Marco Behrendt and Liam Comerford.
 
This research was funded by the Engineering \& Physical Sciences Research Council (EPSRC) with grant no. EP/R006768/1.

 \bibliographystyle{plain}
 \bibliography{biblio.bib}

\begin{thebibliography}{1}

\bibitem{intervalfourier2020}
Marco {De Angelis}, Marco Behrendt, Liam Comerford, Yuanjin Zhang, and Michael
  Beer.
\newblock {Forward interval propagation through the discrete Fourier
  transform}.
\newblock In {\em The 9th international workshop on Reliable Engineering
  Computing}, pages 39--52. 2020.

\bibitem{kaucher1980interval}
Edgar Kaucher.
\newblock Interval analysis in the extended interval space {IR}.
\newblock In {\em Fundamentals of Numerical Computation (Computer-Oriented
  Numerical Analysis)}, pages 33--49. Springer, 1980.

\bibitem{kearfott2005standardized}
Ralph~Baker Kearfott, Mitsuhiro~T Nakao, Arnold Neumaier, Siegfried~M Rump,
  Sergey~P Shary, and Pascal Van~Hentenryck.
\newblock Standardized notation in interval analysis.
\newblock In {\em Proc. XIII Baikal International School-seminar
  “Optimization methods and their applications}, volume~4, pages 106--113.
  Citeseer, 2005.

\bibitem{moore1966interval}
Ramon~E Moore.
\newblock {\em Interval analysis}, volume~4.
\newblock Prentice-Hall Englewood Cliffs, 1966.

\bibitem{neumaier1990interval}
Arnold Neumaier and Arnold Neumaier.
\newblock {\em Interval methods for systems of equations}.
\newblock Number~37. Cambridge university press, 1990.

\bibitem{wiki:minkowski_addition}
{Wikipedia contributors}.
\newblock Minkowski addition --- {W}ikipedia{,} the free encyclopedia, 2021.
\newblock [Online; accessed 18-December-2021].

\end{thebibliography}

\end{document}